\documentclass[12pt,leqno]{article}
\usepackage{amsthm,amssymb,amsmath,amsfonts}
\usepackage{txfonts}

\newtheorem*{mtheorem}{Main Theorem}
\newtheorem*{ftheorem}{Final Theorem}

\newtheorem*{corollary}{Corollary}
\newtheorem*{proposition1}{Proposition 1}
\newtheorem*{proposition2}{Proposition 2}
\newtheorem*{proposition3}{Proposition 3}

\newtheorem*{lemma}{Lemma}

\newcommand {\bs}      {\bigskip}
\newcommand {\ms}      {\medskip}
\newcommand {\hook}    {\hookrightarrow}

\newcommand {\F}       {{\mathcal F}}
\newcommand {\ti}      {\widetilde}
\newcommand {\TD}      {\ti{A_0\cap A_1}}
\newcommand {\e}       {\varepsilon}

\newcommand {\bsk}     {\bigskip}

\newcommand {\bx}      {\hspace{10mm}$\blacksquare$}

\newcommand {\nn}      {\nonumber}
\newcommand{\be}          {\begin{eqnarray}}
\newcommand{\eee}         {\end{eqnarray}}
\newcommand{\bel}[1]      {\begin{equation}\label{#1}}
\newcommand{\ee}          {\end{equation}}
\newcommand {\SECT}[1]
{\section*{\centerline{\normalsize {\bf #1}}}}

\begin{document}
\parindent 1em
\parskip 0mm
\centerline{{\bf A NEW APPROACH TO INTERPOLATION OF COMPACT LINEAR OPERATORS}}
\vspace*{6mm}
\centerline{By~ {\sc Evgeniy Pustylnik}}\vspace*{2mm}
\centerline {\it Department of Mathematics, Technion - Israel Institute of Technology,}\vspace*{2mm}
\centerline{\it 32000 Haifa, Israel}\vspace*{2mm}
\centerline{\it e-mail: evpust@gmail.com}
\bsk\bsk
\renewcommand{\thefootnote}{ }
\footnotetext[1]{ 
\hspace*{6mm} {\it Key Words and Phrases:} Compact operators, interpolation, basic sequences}

\begin{abstract} We prove an abstract theorem on keeping the compactness property of a linear operator after interpolation in Banach spaces.
Our approach consists of two features. Applying the principle ``reductio ad absurdum" we obtain a possibility to carry out
all proofs only for some specially constructed subspaces of the given spaces, e.g. having a common Schauder basis. As a
second feature, we consider in all assertions only embedding operators obtaining the full result just at the end of the
paper. No analytical presentation of operators, spaces and interpolation functors is required and the complex method is
admissible as a particular case.
\end{abstract}
\SECT{Introduction}


\indent
\par Interpolation theory of linear operators was created more than 60 years ago for obtaining properties of such
operators on intermediate spaces of some scale when these properties are known at the endpoint spaces. The scales
were ordered by inclusion with the help of numerical or functional parameters like the Lebesgue spaces $L_p$,
Sobolev spaces $W_p^k$ or spaces $\F_{[\theta]}$, defined by the complex methods. Later on the scales were
replaced by general collections of intermediate spaces for various Banach couples $\vec A=(A_0,A_1)$. Namely,
each of such spaces $A$ must satisfy the embeddings $A_0\cap A_1\hook A\hook A_0+A_1$ where
$$ \|a\|_{A_0\cap A_1}\!=\max\{\|a\|_{A_0},\|a\|_{A_1}\},\quad \|a\|_{A_0+A_1}\!=\inf\{\|a_0\|_{A_0}\!+
   \|a_1\|_{A_1},\ a=a_0+a_1\} $$
and the symbol ``$\hook$" denotes a linear and continuous embedding.

We recall that a rule $\F$, which to any Banach couple $\vec A=(A_0,A_1)$ assigns some intermediate space
$A=\F(\vec A)$, is called an {\it interpolation functor} if, for any other couple $\vec B=(B_0,B_1)$ and any
linear operator $T:A_0\to B_0,\ A_1\to B_1$, this operator is bounded from $A$ to $B=\F(\vec B)$.
Using the closed graph theorem, one can show existence of a constant $C$ such that
\bel{E1}
\|T\|_{A\to B}\le C\max\{\|T\|_{A_0\to B_0},\,\|T\|_{A_1\to B_1}\}
\ee
for any $T$ as above. For more details and definitions of interpolation theory, we refer to the monograph [1].

The problem on interpolation of compactness property is the oldest one  and, from the very beginning, it turned out
to be the most difficult and challenging. The proposed solutions either used some special properties of considered
spaces or constructed various approximation hypotheses in general case. And with all of this, for the complex method
of interpolation (for example), this problem is still far from complete solution.

In the present paper we do not specify any spaces, operators or functors and use only some rarely encountered
properties of bases in Banach spaces. For simplicity, we do not consider cases where some of interpolated couples are
reduced to one space, since such a situation is now investigated as fully as possible (see, e.g., [3], Theorems 2.1 and
2.2). We also exclude other trivial couples, mentioned in [1], Section 2.2.{\bf B}. The general result of our paper
asserts that the interpolation of compact operators is not a separate problem at all, since {\sl the compactness is
interpolated automatically any time when one has a standard interpolation of bounded operators}.

This claim concerns, first of all, operators which are compact at the both endpoint spaces, like the interpolation
of bounded operators (the so-called ``two-sided" compactness). If an operator is compact only at one of endpoints
(``one-sided" compactness), we need an additional condition (2), explained below. Moreover, in the last
section of this paper, we give some necessary and sufficient conditions, replacing (2). \ms

\SECT{Solution of the compactness problem}


\indent
\par {\bf 1.} We start our investigations with embedding operators, since the consequent pass to arbitrary operators was
already considered by M. Cwikel, N. Krugljak and M. Masty\l o in [4] without solution of the general problem.
Since their proof deals only with the complex method and the case where one endpoint of the interpolated couple is
not changed, we repeat it in the last section in a general form.

So, let $(A_0,A_1),\,(B_0,B_1)$ be two nontrivial Banach couples with an interpolation functor $\F$,
defined on them. Let $A_0\hook B_0,\,A_1\hook B_1$ and $J$ be the corresponding embedding operator. Thus,
due to interpolation, $J$ gives an embedding $\F(A_0,A_1)=A\hook B=\F(B_0,B_1)$. For
simplicity, we assume that $\|J\|\le1$, i.e. $\|x\|_{B_i}\le \|x\|_{A_i},\ i=0,\,1$. The problem is to prove
compactness of the embedding $A\hook B$ if such is the embedding $J: A_1\to B_1$ (``one-sided" compactness).

Of course, the ``one-sided" compactness of embedding for an intermediate space may be impossible if this space is
not sufficiently ``distant" from the second space of the couple, where the compactness is not given. This imposes
some additional restrictions on the interpolation functor, which must contain some kind of measure of admissible
``distance". For example, those restrictions are shown as necessary for some cases of real interpolation in the
papers [2], [5] etc.

In our paper such a restriction will be used in the following form:

{\sl There exists a function $W(\alpha,\beta),\
\alpha,\beta>0$, with $\lim_{\beta\to 0}W(\alpha,\beta)=0$ when $\sup\alpha<\infty$, such that
\bel{E2}
\|T\|_{A\to B}\le W(\|T\|_{A_0\to B_0},\|T\|_{A_1\to B_1}) \ee
for any linear operator} $T:A_i\to B_i\ (i=0,1)$.
For example, in the case of a complex interpolation functor $\F_{[\theta]},\ 0<\theta<1$, such a
function exists in the form $W(\alpha,\beta)=\alpha^{1-\theta}\beta^\theta$. \ms

\begin{mtheorem} Under conditions on the spaces $(A_0,A_1),\,(B_0,B_1)$ and the functor $\F$ posed above,
the embedding $A\hook B$ is compact.
\end{mtheorem}

We shall use the following standard way for beginning the proof. We suppose, on the contrary, that the embedding $A\hook B$
is not compact. Then there exists a bounded infinitely dimensional sequence $(x_n)\subset A,\ x_n\neq0$, with no convergent
subsequence in $B$. It follows immediately from this that the sequence of the norms $\|x_n\|_B$ is bounded from
below too, namely, $\inf\|x_n\|_B>0$, since every unbounded from below set necessarily contains a sequence converging to zero.
The infinite dimension of the sequence $(x_n)$ also follows from its non-compactness in $B$, because each bounded set with
finite dimension is compact in any topology.

Surely, there are infinitely many such sequences for any two spaces with non-compact embedding $A\hook B$. Each of these sequences
will be called below {\it supporting sequence} of the given embedding. Due to embedding inequalities, we obtain the following relations:
\bel{E3} 0<\inf\|x_n\|_B\le\inf\|x_n\|_A,\qquad \sup\|x_n\|_B\le\sup\|x_n\|_A<\infty.
\ee

Let us point out also the following two properties:

1) any subsequence of the supporting sequence is supporting as well for the same embedding;

2) the elements of any supporting sequence have equivalent norms in the spaces $A$ and $B$, i.e. $\|x_n\|_A\sim\|x_n\|_B$. \ms

Now we show that a supporting sequence can be chosen consisting only of elements from $A_0\cap A_1$. It is enough for this to show
that the number of such elements, belonging to $A_0\cap A_1$, is infinite. Otherwise, infinitely many elements from our supporting
sequence may belong to the spaces $A_0$ or $A_1$ separately. But the space $A_1$ is embedded into $B_1$ compactly, that is, any
infinite amount of elements from $A_1$ has there a converging sequence. Otherwise, if these extra elements are all in $A_0$ (and also
in $A$, because they belong to the supporting sequence), we can choose an operator $T$ equal zero on the whole $A_1$, but not on those
extra elements from $A$. In this case we obtain a contradiction to the inequality (2).

Thus we obtained possibility to take all supporting sequences from $A_0\cap A_1$. Because in all future discussions, we shall
use only elements from these sequences, we may (and will) suppose that the space $A_0\cap A_1$ is dense in either of spaces $A_0, A_1$
so that they are regular (see [1], page 116). Consequently, the space $A$ is also regular as interpolation one. Remark by the way
that many interpolation methods (for example, the most important for us complex method $\F_{[\theta]}$) give regular interpolation
spaces by definition. Therefore, for them, our explanations before may be omitted.

Another important role in our following proofs will be played by the concept of {\it basic sequences}. Recall that a sequence
$(x_n)$ in a Banach space $E$ is said to be basic one if $(x_n)$ is a Schauder basis of its closed linear hull $[x_n]_E$. (The
subscript $E$ is added here for stressing the norm in which the closure of a hull of $(x_n)$ is taken, since the same sequence may be
considered as a basic one in different spaces.) Like for the supporting sequences, the basic sequence also may be reduced, remaining to
be a basis, but in a new, reduced space (see [7], Proposition 4.1(a), page 27). \bs

{\bf 2.} Now we choose some supporting sequence $(x_n)\subset A$ as initial set for all our following constructions. As shown in
Theorem 1.3 ([8], page 53), the inequalities (3) above ensure that $(x_n)$ has a basic subsequence $(x_{n_k})$ always when $(x_n)$
does not have any weak limiting point $x_0\neq 0$. Otherwise, if such a weak limiting point $x_0$ does exist, we obtain a slightly
another (shifted) basic sequence $(x_{n_k}-x_0)$ which also belongs to $A$ and satisfies inequality (3) even if $x_0\notin (x_n)$.
Moreover, the sequence $(x_n-x_0)$ is supporting itself, hence we needn't separate this case always below. Anyway, we may conclude
that each supporting sequence $(x_n)$ generates (immediately or after shift) a basic subsequence, satisfying (3), which will be
denoted by $(y_n)$. \ms

For the next step of our proof, we refer to the mighty Corollary 1.4 ([8], page 56): {\sl the existence of a basic subsequence
$(x_{n_k})$ for arbitrary given sequence $(x_n)$ from an arbitrary given Banach space $E$ is equivalent to existence of a total
space $V\subset E^*$ such that}
\bel{E4} \lim_{n_k\to\infty}\frac{f(x_{n_k})}{\|x_{n_k}\|_E}=0
\qquad \text{for all $f\in V$}.
\ee
We shall use this assertion in two opposite directions. For the space $A$, we have got already the basic sequence $(y_n)$,
following from its supporting sequence. Therefore the relation (4) is fulfilled after replacing $(x_{n_k})$ by $(y_n)$ and $E$
by $A$. Passing to the space $A_0\cap A_1$, we recall that it is dense in the regular space $A$ and thus $A^*\subset(A_0\cap A_1)^*$.
Consequently, we may use the same set $V$ for the space $A_0\cap A_1$ as well. Moreover, we get the inequality
$$
\|y_n\|_{A_0\cap A_1}
\ge\alpha\|y_n\|_A, \ ~~~n=1,2,\dots,
$$
with some embedding constant $\alpha$ that additionally confirms (4) with new sequence $(y_n)$ and
space $A_0\cap A_1$. Using the opposite direction of the same assertion, we obtain the existence of a basic sequence (say $(y_{n_k})$)
in the space $A_0\cap A_1$. \ms

The reader could easily see that the basic sequences, defined in the Section 2, are included one into another:
$$ (y_n)=(x_{n_k})\subset(x_n), \qquad (y_{n_k})\subset(y_n). $$
Thus $(y_{n_k})$ is included into all other sequences and may be taken as a basic one in various spaces, for example, in $A$ and $A_0\cap A_1$.
Let us re-denote $z_n:=(y_{n_k})$ and define
$$
\ti A:=[z_n]_A,\ \TD:=[z_n]_{A_0\cap A_1}.
$$
Analogously we shall use the sign ``tilde" for other
spaces with bases reduced from $(z_n)$ (for simplicity, we sometimes omit the words ``reduced from"). \ms

\begin{lemma} Let the spaces $A, B$ be defined on the base of some common variety of elements $\{x\}$ so that their sum $A+B$ may be
considered. If these spaces have a common basis $(z_n)$ then all basic coordinates of a given element $x$ are the same.

Namely, if $x=\sum_{n=1}^\infty a_nz_n$ in the space $A$ and $x=\sum_{n=1}^\infty b_nz_n$ in the space $B$,
then $a_n=b_n$ for all $n$.
\end{lemma}

\par {\it Proof.} Consider difference of the representations of $x$:\ \ $0=\sum_{n=1}^\infty (a_n-b_n)z_n$. \ But a zero-element may have
only zero-coordinates in any basis.

Thus $a_n=b_n$ for all $n$ as required.\bx

\begin{corollary} The same fact occurs for finite (partial) sums $\sum_{k=1}^n a_kz_k$ with any given $n$.
\end{corollary} \ms

{\bf 3.} Let us start to investigate the basic properties of the sequence $(z_n)$ for the spaces $A_0$ and $A_1$ individually. \ms

\begin{proposition1} The sequence $(z_n)$ is basic one in either of spaces $A_0$, $A_1$, defining the spaces with a basis $\ti A_0=
[z_n]_{A_0}$ and $\ti A_1=[z_n]_{A_1}$.
\end{proposition1}

\par {\it Proof.} Denote by $s_n(x)$ the partial sums of the known already basic expansion in the space $\TD$, namely,
$s_n(x)=\sum_{k=1}^n z_kf_k(x)$. The Theorem 4.1 from [7, page 25] asserts that $(z_n)$ is a basis in a given space $E$ if and only if
$E=[z_n]_E$ and the corresponding operators $s_n$ are bounded in $E$ on the whole, i.e. if all $\|s_n\|_E\le M$ for some constant $M$.
For example, such a constant exists for the space $\TD$. Our problem is to state the existence of analogous bounds for the spaces
$\ti A_0$ and $\ti A_1$.

Due to full symmetry of spaces $\ti A_0, \ti A_1$, we proceed our proof for one of these spaces, say, for $\ti A_0$. Using the Hahn-Banach
theorem, we may extend any of functionals $f_k(x)$ to the whole space $\ti A_0$. Moreover, due to density of $\TD$ in $\ti A_0$, these
extensions are unique, defining (as sums) continuous linear operators $s_n(x), n=1,2,\ldots$ on the whole $\ti A_0$.

Any continuous linear operator is bounded, namely, $\|s_n(x)\|_{A_0}<\infty$ for each $x\in\ti A_0$ like the case of $x\in\TD$.
A pass from individual to global boundedness of operators is the matter of so called ``principle of uniform boundedness" which for
Banach spaces is proved, e.g., as ``Theorem of Banach-Steinhaus" in the monograph [6, Ch.III, \S4]. Below we shall use some details
of this proof.

Suppose that the operators $s_n(x), n=1,2,\dots$ are not uniformly bounded in $\ti A_0$ on the whole.
This implies that for any natural $k$ there exist a point $x_k\in\ti A_0$ and an operator $s_{n_k}$ such that $\|s_{n_k}(x_k)\|_{A_0}>k$.
Due to density of $\TD$ in $\ti A_0$ and continuity of the operator $s_{n_k}(x)$, there exists a point $x'_k\in\TD$ so
close to $x_k$ as to satisfy the similar inequality $\|s_{n_k}(x'_k)\|_{A_0}>k$. But the norms in $A_0$ are not greater than the norms
in $A_0\cap A_1$, hence $\|s_{n_k}(x'_k)\|_{A_0\cap A_1}>k$ as well.

The main fact from the above mentioned proof is that any family of linear operators, being unbounded in some space $E$, is also
unbounded on each ball from this space---for example, on unit one. Therefore we may take all $x'_k$ with $\|x'_k\|_{A_0\cap A_1}\le 1$,
getting that $\|s_{n_k}\|_{A_0\cap A_1}\longrightarrow\infty$ for $k\to\infty$ in the contrast to existence of a basis in $\TD$. This
proves the existence of a basis in the space $\ti A_0$ as required. The case of $\ti A_1$ is similar.\bx \ms

\begin{proposition2} The intersection of spaces $\ti A_0$ and $\ti A_1$ coincides with $\TD$.
\end{proposition2}
\par {\it Proof.} We already have seen how the spaces $\ti A_0$ and $\ti A_1$ may be obtained via passing to limits from $\TD$ by
the norm of $A_0$ or $A_1$. Recall that $\|x\|_{A_0\cap A_1}=\max\{\|x\|_{A_0},\|x\|_{A_1}\}$. Therefore the norms in $A_0\cap A_1$ can be
only in two forms: $\|x\|_{A_0}$ or $\|x\|_{A_1}$ and the limiting process in either of spaces $A_0, A_1$ gives different results. When we
consider the intersection of the resulting spaces $\ti A_0$ and $\ti A_1$, all their acquisitions disappear as not common and remain only
initial elements from $\TD$. Thus we obtain a proof of the equality $\ti A_0\cap\ti A_1=\TD$.

Of course, it may be that $\|x\|_{A_0}=\|x\|_{A_1}$, but this ``equilibrium" disappears immediately with start of the limiting process for
one of norms.\bx \ms

Let us start to exploit some standard properties of the spaces $\ti A_0$ and $\ti A_1$ as having a proper basis.
For arbitrary $n=1,2,\dots$, we define finite dimensional operators $P_nx=\sum_{k=1}^n a_kz_k$, where $x\in \ti A_0$ or
$x\in \ti A_1$ and $a_k$ are the basic coordinates of the element $x=\sum_{k=1}^\infty a_kz_k$. (Here we change $(z_n)$
to $(z_k)$ in order not to confuse with numeration of operators $P_n$.) Recall, especially, that $a_k$ are the same in
all spaces with common basis. Due to properties of bases, we have got a constant $C\ge1$ such that $\|P_nx\|_{\ti A_i}\le
C\|x\|_{\ti A_i}, i=0,\,1$, for all $n$ and for all $x\in\ti A_0$ or $x\in\ti A_1$, that is, the operators $P_n$ are
uniformly bounded. Moreover, $\lim_{n\to\infty}\|P_nx-x\|_{\ti A_i}=0$ if $x\in\ti A_i,\ i=0,\,1.$ \ms

 {\bf 4.} For the following discussion, we have to return explicitly to the embedding operator $J$ which is compact as an
 operator from $\ti A_1$ to $B_1$. Recall that $J$ is not a usual operator and its applying has some peculiarities. For example, it needn't be mentioned in a simple transfer from one space to another when it is equivalent to identity:\ $Jx=Ix=x$. Such a situation was, e.g., in the formula (3) above, where we wrote $\|x_n\|_A$ and $\|x_n\|_B$ for the same element $x_n$ in different spaces. Of course, using an embedding operator becomes convenient for topological expression with convergence, compactness and so on. Let us agree to use the symbol $J$ only within the spaces of the type $B$.

Recall that the space $\ti A_1$ as a subspace of $A_1$ proceeds to be embedded compactly into $B_1$. Hence, taking $S$ to be
a unit ball of $\ti A_1$, we obtain a compact set $JS\subset B_1$. For arbitrary given number $\e>0$, we may construct
an epsilon-net in $JS$ as a collection of cells $\{g_1,\ldots,g_m\}$, having radii no greater than $\e$ and covering all $JS$.
Denote by $\{\bar x_k\}$ the elements $x\in S$, such that the corresponding $J\bar x_k$ lie in all possible cells of the
net, but only one by one. Clear, the number of such elements will be no greater than $m$. \ms

\begin{proposition3} Let $n\to\infty$. Then $J(P_n-I)\to 0$ uniformly by the norm of operators acting from $\ti A_1$ to $B_1$.
\end{proposition3}

\par {\it Proof.} Let a number $\e$ be given and the epsilon-net with the points $\{\bar x_k\}$ be as above. Consider $(y_n)=(Jz_n)$ as a sequence of elements in the space $B_1$. Following terminology of [7] (\S8, page 69), we obtain that the sequence $(z_n)$ {\it strictly dominates} the sequence $(y_n)$ due to a linear mapping $J$ connecting them. Then b)$2^\circ$ from the same \S8 of [7] gives that the series $\sum_{k=1}^\infty a_ky_k$ converges in $B_1$ like the series $\sum_{k=1}^\infty a_kz_k$ in $\ti A_1$ with the same coefficients $(a_k)$. For more comprehension, we may denote $\ti B_1=J(\ti A_1)$, though we need from the $B_1$ only its norm.

Since the basic coefficients remain in $\ti B_1$ the same as in $\ti A_1$, we obtain the same operators $P_ny=\sum_{k=1}^na_k y_k$ with the same properties $\|P_ny\|_{\ti B_1}\le C\|y\|_{\ti B_1}$ and $\lim_{n\to\infty}\|P_ny-y\|_{\ti B_1}=0$ if $y\in\ti B_1$. Due to linearity of the embedding operator $J$, we may write that $P_ny=P_nJx=\sum_{k=1}^na_kJz_k=J(\sum_{k=1}^na_kz_k)=JP_nx$. In result we obtain a crucial inequality for our proving: \ $\|JP_nx\|_{\ti B_1}\le C\|Jx\|_{\ti B_1}$. Finally, taken $x-\bar{x_k}$ instead of $x$, we come to the inequality
$$ \|JP_nx-JP_n\bar{x_k}\|_{\ti B_1}\le C\|Jx-J\bar{x_k}\|_{\ti B_1} $$
again due to linearity of all operators.

Thus the inequality $\|Jx-J\bar x_k\|_{B_1}\le\e$ provides that
$\|JP_nx-JP_n\bar x_k\|_{B_1}\le C\e$ for all $k$. These inequalities both are universal, independent of $n,x$ and
$\bar x_k$ if we agree always to choose the point $J\bar x_k$ as the closest one to $Jx$ among all possible $k$. However
the last needed inequality $\|JP_n\bar x_k-J\bar x_k\|_{B_1}\le\e$ does depend on $\{\bar x _k\}$, requiring that
$n\ge n_k$. Fortunately, the amount of all $k$ depends only on $\e$, and we may require that $n\ge n_\e=\max n_k$.
Joining all three inequalities with the same $x$ into one, we obtain that
\be
\|JP_nx-Jx\|_{B_1}&\le&\|JP_nx-JP_n\bar x_k\|_{B_1}
\nn\\
&+&\|JP_n\bar x_k-J\bar x_k\|_{B_1}+\|Jx-J\bar x_k\|_{B_1}\le (C+2)\e
\nn
\eee
for any $x$ from the unit ball of $\ti A_1$ if $n$ is sufficiently big. Consequently, the Proposition 3 is proven.\bx \ms

{\bf 5.} Now we are able to finish the proof of the Main Theorem, using the inequality (2). Unfortunately, this inequality is
postulated above only for linear operators, defined and continuous on the whole spaces $A_0, A_1$, whereas our Proposition 3
deals with the rather smaller spaces $\ti A_0, \ti A_1$. At the same time, a simultaneous extension of the operators $P_n$
to the whole space $A_1$ (with preservation of norms) may be impossible, even despite of their finite dimension, if the
complement $A_1\backslash\ti A_1$ is not closed in $A_1$.

In order to avoid this problem, we propose to pass to one-dimensional operators $T_n=P_n-P_{n-1}$. Every one-dimensional operator
is equivalent to some linear functional, differing from it by addition of some constant vector-multiplier. Therefore we may apply
the Banach-Hahn theorem, obtaining possibility to consider any $T_n$ on the whole $A_0, A_1$ (although these extensions
may differ outside of $\ti A_0, \ti A_1$). Moreover, because the norms of extended operators did not grow, we can use all estimates
of their norms for needed proofs. For example, we get that $\|T_n\|_{A_0\to B_0}\le 2C$ for $n=1,2,\dots$ Still more valuable
is that due to Proposition 3
 $$ \|T_n\|_{A_1\to B_1}=\|T_n\|_{\ti A_1\to B_1}\le\|P_n-I\|_{\ti A_1\to B_1}+\|P_{n-1}-I\|_{\ti A_1\to B_1}\longrightarrow 0 $$
as $n\to \infty$. Only now the operators $T_n$ became suitable for applying the inequality (2) which gives that
$$ \lim_{n\to\infty}\|T_n\|_{A\to B}=0. $$

Apply now a given operator $T_n$ to the basic element $z_n$ with the same number. Since $P_{n-1}z_n=0$ and $P_nz_n=z_n$, we obtain that
\bel{E5}
\|z_n\|_B=\|T_nz_n\|_B\le\|T_n\|_{A\to B}\|z_n\|_A\longrightarrow 0\quad\text{as $n\to \infty$}, \ee
because all $z_n\in (x_n)$ and $\sup \|x_n\|_A<\infty$ by the right-hand part of inequalities (3).

It remains to say that the relation (5) contradicts to the left-hand part of inequalities (3), and this contradiction proves the
Main Theorem. \ms

{\it Remark}. To our knowledge, an approach to connection between compactness and bases in Banach spaces was firstly described
in the monograph [6] (Ch.V, Theorem 3) but only for one space, without embeddings and interpolation. \ms

\SECT{Supplements}


\indent
\par {\bf 1.} The same proof is applicable to proving interpolation of the ``two-sided" compactness of linear operators without
inequality (2). Namely, if both embeddings $A_i\hook B_i,\ i=1,2$, are compact, we should show that $\|J(P_n-I)\|_{\ti A_0\to B_0}
\rightarrow 0$\ like in the case of the spaces $\ti A_1, B_1$ before, extend operators $T_n$ to the whole space $A_0$ and use the
inequality (1) for estimating the norms $\|T_n\|_{A\to B},\ n=1,2,\dots$ This will lead us to the same contradiction in estimation
of the norms $\|z_n\|_B$. \ms

{\bf 2.} When starting this paper with embedding operators solely, we have promised to consider the general situation later. Now we
want to keep our promise. \ms

Let $T$ be an arbitrary linear operator, acting continuously
from $A_0$ to $B_0$ and compactly from $A_1$ to $B_1$. Define two spaces $Y_i\hook B_i\ (i=0,1)$ consisting of
all $y=Tx,\ x\in A_i$, with norms $\|y\|_{Y_i}=\inf\bigl\{\|x\|_{A_i}\ :\ y=Tx\bigr\}$. As known, every such a space is
called {\it image} of $T$ in the corresponding space $B_i$ and is a Banach space. Present now $T$
as a composition $T=JS$, where the operator $S$ is defined by $Sx=Tx$ for all $x\in A_0+A_1$ but considered as a bounded
operator from $A_i$ to $Y_i,\ i=0,1$. Consequently $J$ will be the embedding operator from $Y_i$ into $B_i$ for the same $i=0,1$.
Applying the functor $\F$, we obtain that $S$ acts boundedly from $A=\F(A_0,A_1)$ into $Y=\F(Y_0,Y_1)$ by standard interpolation
theorems. At the same time, $J$ appears to be compact as an operator from $Y_1$ into $B_1$ due to compactness of the
operator $T$ and thus it is compact as embedding from $Y=\F(Y_0.Y_1)$ into $B=\F(B_0,B_1)$ in virtue of the Main Theorem, proved
above. In result, $T$ is also compact as an operator from $A$ to $B$. \ms

{\bf 3.} {\it Necessity of conditions}. Let us return to ``one-sided" compactness and to embedding operators. As the reader could
see, the principal point of our proof is the inequality (2). Of course, this inequality is rather simple for checking. Moreover,
it is peculiar to many classical interpolation methods. Especially important is that it enables to solve a very long standing
problem on interpolation of compactness by the complex method. However, it is reasonable to ask, to what extent this inequality
is necessary in general? An immediate answer after our proof is: ``the inequality (2) is excessive". Indeed, generally speaking,
we need not know any special analytical form of the function $W(\alpha,\beta)$. We even need not know whether such a function really
exists outside of the basic operators $P_n$. And the last surprise is that we may consider only one-dimensional operators.

Well! All the listed surplus may be excluded from the minimal sufficient conditions, but what else? Could we give any really
``necessary and sufficient" conditions for interpolation of compactness property in all cases? By the way, we need not take a carry
of interpolation itself---it is given from the very beginning. \ms

\begin{ftheorem} Let the interpolation functor $\F$ and the Banach couples $(A_0,A_1)$, $(B_0,B_1)$ be defined as in the Main
Theorem above (without inequality (2)). Let, in addition, the spaces $A_0, A_1$ be regular, namely, the intersection $A_0\cap A_1$ be
dense in each of these spaces. Then for the embedding $A\hook B$ to be compact, it is necessary and sufficient that any bounded
basic sequence $(z_n)$ from $A_0\cap A_1$, converging to zero in the space $B_1$, did the same in the space $B$. \end{ftheorem}

{\it Proof.} In order to prove the sufficiency, it is enough to repeat the proof of the Main Theorem without any change. Only the last,
crucial contradiction will be obtained now directly, without the inequality (2). For the proof of necessity, we must explain, first of all,
how we could test the basic sequences which are not given before and could be absent at all. This may be a real problem for computations,
but theoretically not, because {\sl at least one basic sequence exists in any Banach space} (see, [8], Theorem 1.1, page 48). Taken from
the intersection, the sequence $(z_n)$ belongs to $A, A_0, A_1$. As in the Main Theorem, we denote the closed hull of $(z_n)$ in every of
these spaces by a tilde above and as in Proposition 1 we prove that $(z_n)$ is a basis in all these spaces. Next we define in these reduced
subspaces the basic operators $P_n,\ n=1,2,\dots$ As shown in Proposition 3,
in any case of compact embedding, these operators uniformly tend to $I$. Consequently, the sequence $(z_n)$ converges to zero.

This remark gives us a key to the end of the proof. A given compactness of the embedding $A_1\hook B_1$ implies that the embedding
$\ti A_1\hook B_1$ is compact as well and gives us a basic sequence converging to zero. By assumption of necessity, the same must occur in
the space $B$ - otherwise the embedding $\ti A\hook B$ cannot be compact. Of course, the embedding of the larger space $A$ into $B$ thus
also cannot be compact. \bx
\bsk\bsk

\centerline{\sc Declarations}

\bsk

\centerline{\bf Funding and Competing interests}

\bsk
$\bullet$ The author has no relevant financial or non-financial interests to disclose.

$\bullet$ The author has no competing interests to declare that are relevant to the content of this article.

\fontsize{10.5}{12}\selectfont

\end{document}